\documentclass[12pt]{article}
\usepackage{ragged2e}
\usepackage{graphicx}
\usepackage{amsmath}
\usepackage{amssymb}
\usepackage{amsthm}
\usepackage{algpseudocode}
\usepackage{array}
\usepackage{caption}
\usepackage{upgreek}
\usepackage{etoolbox}
\usepackage{mathtools}
\usepackage{relsize}
\usepackage{setspace}
\usepackage{comment}
\usepackage{mathrsfs}
\usepackage{bm}
\usepackage{dutchcal}
\usepackage{dsfont} 
\usepackage{float} 

\usepackage{indentfirst} 
\usepackage{calligra}  
\usepackage[T1]{fontenc}
\usepackage{chancery}
\usepackage[T1]{fontenc}

\usepackage{libertine}

\usepackage{enumitem}
\usepackage{changepage}
\usepackage{setspace}
\usepackage{geometry}
\usepackage{hyperref}
\hypersetup{
    colorlinks=true,
    linkcolor=purple,
    filecolor=magenta,      
    urlcolor=cyan,
    citecolor=blue,
}
\usepackage[noabbrev,capitalize,nameinlink]{cleveref}

\theoremstyle{definition}
\newtheorem{defi}{Definition}[section]

\theoremstyle{definition}

\theoremstyle{definition}

\theoremstyle{definition}

\theoremstyle{definition}
\newtheorem{rmk}[defi]{Remark}

\theoremstyle{definition}

\theoremstyle{definition}
\newtheorem{eg}[defi]{Example}

\theoremstyle{definition}

\theoremstyle{definition}

\theoremstyle{plain}
\newtheorem{lem}[defi]{Lemma}

\theoremstyle{plain}
\newtheorem{prop}[defi]{Proposition}

\theoremstyle{plain}

\theoremstyle{plain}
\newtheorem{thm}[defi]{Theorem}

\theoremstyle{plain}

\theoremstyle{plain}
\newtheorem{cor}[defi]{Corollary}

\newcommand{\z}{\mathbb{Z}}

\newcommand{\cc}{\mathbb{C}}
\newcommand{\pp}{\mathbb{P}}

\newcommand{\oo}{\mathscr{O}}

\renewcommand{\tau}{\uptau}

\DeclareMathOperator{\spec}{Spec}
\DeclareMathOperator{\proj}{Proj}

\newcommand{\wt}{\mbox{\bf wt}}
\newcommand{\gitquo}{/\!\!/} % GIT quotient

\newlength\mlen
\newcommand{\m}[1]{\makebox[\mlen]{$#1$}}

\usepackage{calc}  % box長度
\newcommand{\len}[2]{\makebox[\widthof{#1}]{\ensuremath{#2}}}

\usepackage{tikz, tikz-cd}
\usetikzlibrary{arrows, shapes, cd, patterns}
\usetikzlibrary{trees, positioning, fit}
\tikzset{
  level 1/.style={level distance=6cm, sibling distance=3.5cm},
  level 2/.style={level distance=6cm, sibling distance=1.3cm}
}

\title{Flips and Flops Constructed by GIT Quotient}
\author{Hung-Pin Chang}
\date{}

\begin{document}

\maketitle

\begin{abstract}
Brown constructed a series of threefold flips given by the GIT quotient of a hypersurface in $\cc^5$. In this article, we classify threefold flips and flops which are the GIT quotients of complete intersections in $\cc^6$. We also show that there are no more new examples as GIT quotients of complete intersections in $\cc^n$ with $n\geq7$.
\end{abstract}

\section{Introduction}
The minimal model program plays a fundamental role in modern algebraic geometry. Roughly speaking, it aims to produce a nice minimal model among birational equivalent classes of algebraic varieties. Flips and flops are elementary birational maps appearing in dimension three or higher. The existence and termination of flips have been in the center field for decades. Also, in case minimal models are not unique, two birationally equivalent minimal models are connected by flops (cf. \cite{Kaw08}). Therefore, a detailed and explicit understanding of flips and flops is very much welcome. 

We first recall the definition of flips and flops. 
\begin{defi}
A threefold \emph{flip} (resp. \emph{flop}) is a diagram $X^-\to X\gets X^+$ of normal complex quasiprojective threefolds satisfying the following conditions:
\begin{itemize}
    \item[$\bm C.$] both morphisms are birational and projective, contracting only finitely many curves $\Gamma\subset X^-$ and $\Gamma'\subset X^+$ to the isolated singular point $P\in Y$.
    \item[$\bm I.$] their canonical divisors $-K^-\coloneqq -K_{X^-}$ and $K^+\coloneqq  K_{X^+}$ are relatively ample (resp. relatively trivial), i.e., $-K^-.\Gamma>0$ (resp. $K^-.\Gamma=0$) for any curve $\Gamma$ contracted by $X^-\to X$ and similarly for $K^+$.
    \item[$\bm S.$] $X^-$ and $X^+$ has only terminal singularities. 
\end{itemize}
\end{defi}
\begin{rmk}
Conditions C, I, and S stands for contraction, intersection, and singularity conditions, respectively.
\end{rmk}

However, there were not so many explicit examples of flips and flops, even in dimension three. 
 In \cite{Br99}, Brown systematically constructed $3$-dimensional examples by GIT quotient, and he finds all possible flips given by a hypersurfaces $(f=0)\subseteq\cc^5$ quotients via $\cc^*$-actions. 

The purpose of this article is two-fold. First, we extend Brown's constructions to codimension $2$ complete intersection in $\cc^6$. By a very similar technique as in Brown's work, we classify flips and flops as GIT quotient, which we will call them {\it GIT flips and GIT flops} for brevity, of codimension $2$ complete intersection in $\cc^6$.  Notice that one can trivially extend a GIT flip or flop in $\cc^n$ to $\cc^{n+1}$ by introducing a new variable $w$ and a new hypersurface involving a linear term in $w$. Therefore, we are only interested in those examples not of this type, which we call them {\it reduced}. 

\begin{thm}
    Any reduced GIT flip or flop of codimension $2$  is one of the following. 
\begin{table}[H]
    \centering
    \begin{tabular}{llll}
        \hline \hline
        \multicolumn{1}{c}{} &
        \multicolumn{1}{c}{monomials in $f_1$} &
        \multicolumn{1}{c}{monomials in $f_2$} &
        \multicolumn{1}{c}{$\cc^*$-action} \\
        \hline \hline
             Flop & $x_1y_1 + g_1(z_1,z_2)$ & $x_2y_i + g_2(z_1,z_2)$ & $(\m{1},\m{1},-1,-\m{1},0,0;0,0)$ \\
             Flip & $x_1y_1+h(x_2^{a_3},x_3^{a_2}) $ & $x_2y_2 + z^m$ & $(a_1,a_2,a_3,-b_1,-a_2,0;a_1-b_1,0)$ \\
        \hline \hline
    \end{tabular}
    \caption{Flips and Flops for codimension $2$ case}
    \label{flipflop6}
\end{table}
\end{thm}

The second purpose of this article is to show that there is essentially no new example of threefold GIT flips or flops in codimension $3$ or higher. 

\begin{thm}\label{thmhc}
    Any reduced threefold GIT flip or flop has codimension $\le 2$.
\end{thm}

Below, we summarize and list singularity types of $X^-$ and $X^+$ of our classification. 

\begin{table}[H]
    \centering
    \begin{tabular}{lll}
        \hline 
        \multicolumn{1}{c}{} &
        \multicolumn{1}{c}{Sing. in $X^-$} &
        \multicolumn{1}{c}{Sing. in $X^+$} \\
        \hline 
            $(1)$ & $cA/a_1$ (qt.), $cA/1$ & $cA/b_1$ (qt.), $cA/a_2$  \\
            $(2)$ & $cA/a_1$ (qt.), $cA/a_2$ & $cA/b_1$ (qt.), $cA/a_2$  \\
            $(3)$ & $cAx/4$ & $cA/2$  \\
            $(4)$ & $cA/a$ & $cA/b$  \\
            $(5)$ & $cD/3$ & $cAx/2$  \\
            $(6)$ & $cAx/4$ & $cA/3$ (qt.), $cA/2$  \\
        \hline 
    \end{tabular}
    \caption{Singularity types of GIT flips for hypersurface}
    \label{singflip5}
\end{table}

\begin{table}[H]
    \centering
    \begin{tabular}{lll}
        \hline 
        \multicolumn{1}{c}{} &
        \multicolumn{1}{c}{Sing. in $X^-$} &
        \multicolumn{1}{c}{Sing. in $X^+$} \\
        \hline 
            $(1)$ & smooth & smooth \\
            $(2)$ & cDV & cDV \\
            $(3)$ & $cA$, $cA$ & $cA$, $cA$ \\
        \hline 
    \end{tabular}
    \caption{Singularity types of GIT flops for codimension $2$}
    \label{singflop6}
\end{table}

\begin{table}[H]
    \centering
    \begin{tabular}{lll}
        \hline 
        \multicolumn{1}{c}{} &
        \multicolumn{1}{c}{Sing. in $X^-$} &
        \multicolumn{1}{c}{Sing. in $X^+$} \\
        \hline 
            $(1)$ & $cA/a_1$ & smooth \\
            $(2)$ & $cA/a_1$ & $cA/b$ \\
            $(3)$ & $cA/a_1$ & $cA/b_1$, $cA/b_2$ \\
        \hline 
    \end{tabular}
    \caption{Singularity types of GIT flips for codimension $2$}
    \label{singflip6}
\end{table}

In \cref{singflop6}, the cDV singularity may be $cA$, $cD$, $cE_6$, $cE_7$, or $cE_8$. For the flips, $X^+$ may have two $cA/m$ singularities and this is not in Brown's list. We construct examples with those singularities and their explicit polynomials and $\cc^*$-actions will be described in \cref{ex1flop6} and \cref{ex2flop6}, as well as \cref{ex1flip6} through \cref{ex4flip6}.

Gavin Brown informed us that Laura Mallinson independently obtained similar results in her 2022 Ph.D. thesis. These results concern GIT flips of codimension $2$ (cf. \cite{Mal22}).

\paragraph{Acknowledgments.} We would like to sincerely thank Professors Jungkai Alfred Chen, Hsueh-Yung Lin, Ching Jui Lai, Jiun-Cheng Chen, Jheng-Jie Chen, and other lab members, for their valuable discussions. We are also grateful to Professor Gavin Brown for informing us about Laura Mallinson's independent work. Our sincere thanks go to Professor Hamid Abban and to Laura Mallinson for their helpful comments and communication. We are thankful to National Taiwan University, the National Center for Theoretical Sciences, the Institute of Mathematics at Academia Sinica, and the Laboratory of Birational Geometry for their support. Finally, we appreciate the funding from grant MOST112-2123-M002-005, generously provided by the Ministry of Science and Technology of Taiwan.

\section{Construction}

Consider a $\cc^*$-action on $\cc^n$ given by
\[
\lambda\cdot (x_1,\ldots,x_r,y_1,\ldots,y_s,z_1,\ldots,z_t) = (\lambda^{a_1}x_1,\ldots,\lambda^{a_r}x_r,\lambda^{-b_1}y_1,\ldots,\lambda^{-b_s}y_s,z_1,\ldots,z_t),
\]
where $\lambda\in\cc^*$, $r+s+t=n$, and $a_i,b_j\in\z_{>0}$; we say that $a_i,-b_j,0$ is the weight of $x_i,y_j,z_k$, respectively. Then $\cc^*$-action give us GIT quotients $V\coloneqq \cc^n\gitquo \cc^*$, $V^-\coloneqq (\cc^n-B^-)\gitquo\cc^*$, and $V^+\coloneqq (\cc^n-B^+)\gitquo\cc^*$, where $B^-\coloneqq (x_1=\cdots=x_r=0)$ and $B^+\coloneqq (y_1=\cdots=y_s=0)$. We denote $(1:0:\cdots:0)\in V$ by $P_1$ for convenience. Similarly, $P_2=(0:1:0:\cdots:0)$, $P_3=(0:0:1:0:\cdots:0)$, etc. 

Let $A\coloneqq (f_1=\cdots=f_{c}=0)\subseteq \cc^n$ with $n=c+4$ be a $4$ dimensional completely intersection, where $f_1, \ldots, f_{c}$ are homogeneous polynomials for the weights of coordinates.  The weights of polynomials $f_k$ is an integer $e_k $ such that  $\lambda\cdot f_k=\lambda^{e_k} f_k$. Then we denote this action on $A$ by $(a_1,\ldots,a_r,b_1,\ldots,b_s,0,\ldots,0;e_1,\ldots,e_{c})$. We define closed $3$-folds
\[
X\coloneqq  A\gitquo\cc^*, \quad X^-\coloneqq (A-B^-)\gitquo\cc^*, \quad \mbox{and} \quad X^+\coloneqq (A-B^+)\gitquo\cc^*,
\]
in $V$, $V^-$, and $V^+$, respectively. 

More precisely, write $I\coloneqq (f_1.\ldots,f_{c})$, and then
\begin{align*}
    X^{\phantom{-}}&=\spec\bigg(\frac{\cc[\mathbf{x},\mathbf{y},\mathbf{z}]}{I}\bigg)^{\cc^*},\\
    X^-&=\bigcup_{i=1}^r X^-_i=\bigcup_{i=1}^r \spec\bigg(\frac{\cc[\mathbf{x},\mathbf{y},\mathbf{z},x_i^{-1}]}{I}\bigg)^{\cc^*}.
\end{align*}
This induces morphisms $\pi_{-, i}: X^-_i\to X$ given by inclusion between invariant rings, and hence we have morphism $\pi_-:X^-\to X$. Similarly, we have the morphism $\pi_+:X^+\to X$. 

We set $C^-\coloneqq \pi_-^{-1}(0)$ and $C^+\coloneqq \pi_+^{-1}(0)$. They can be understood as 
\[
C^-
=\proj\frac{\cc[\mathbf{x},\mathbf{y},\mathbf{z}]}{(I,\mathbf{y},\mathbf{z})}
=\proj \frac{\cc[\mathbf{x}]}{(f_1(\mathbf{x},\mathbf{0},\mathbf{0}), \ldots,f_c(\mathbf{x},\mathbf{0},\mathbf{0}) )};  
\]
\[ 
C^+
=\proj\frac{\cc[\mathbf{x},\mathbf{y},\mathbf{z}]}{(I,\mathbf{x},\mathbf{z})}
=\proj \frac{\cc[\mathbf{y}]}{(f_1(\mathbf{0},\mathbf{y},\mathbf{0}), \ldots, f_c(\mathbf{0},\mathbf{y},\mathbf{0}))}.
\]

\section{Numerical Conditions}

In Brown's work, he deduced some numerical conditions for GIT flips of codimension $1$. (See Proposition $1$ and Corollary $3$ in \cite{Br99}.) For example, Brown used the classification of terminal singularities to check that $P_i$ in $X^-$ are terminal, which forces certain monomials to appear in $f_i$ and gives some relations between weights. We first extend these conditions to higher codimensional situations. 

\begin{prop}[C]\label{propc}
    The weights given above have at least two positive weights and at least two negative weights. More precisely, 
    \begin{itemize}
        \item[$(i)$]for $n=5$, the weights is one of the following forms
        \[
        (+,+,+,-,-;+), \; (+,+,-,-,0;0), \; \text{or} \; (+,+,-,-,-;-).
        \]
        \item[$(ii)$]for $n=6$, the weights is one of the following forms
        \[
        \begin{array}{cc}
            (+,+,+,+,-,\m{-};+,\m{+}), 
            & (+,+,-,-,\m{-},\m{-};-,\m{-}), \\ 
            (+,+,+,-,-,\m{0};+,\m{0}),  
            & (+,+,-,-,\m{-},\m{0};-,\m{0}), \\ 
            (+,+,+,-,-,\m{-};+,\m{-}), 
            & (+,+,-,-,\m{0},\m{0};e_1,e_2),  
        \end{array}
        \]
        where $(e_1,e_2)$ may be $(0,0)$ or $(+,-)$. 
    \end{itemize}
\end{prop}
\begin{proof}
    Since $\pi_-$ only contracts finitely many curves, the dimension of $C^-$ is $1$, i.e.,
    \[
    \dim \frac{\cc[\mathbf{x}]}{(f_1(\mathbf{x},\mathbf{0},\mathbf{0}),\ldots,f_c(\mathbf{x},\mathbf{0},\mathbf{0}))}
    = 2.
    \]
    Thus, $r\geq2$. Similarly, $s\geq2$. When $r>2$, then there exist at least $r-2$ polynomials $f_k$ such that $f_k(\mathbf{x},0,\ldots,0)\neq0$. This implies that $f_k$ contains monomial in $\cc[\mathbf{x}]$, so $e_k>0$. The rest of the cases are similar. 
\end{proof}

\begin{cor}\label{corc}
    Suppose that $n\leq6$. Then $f_i$ contains a monomial in $\cc[z_1,\ldots,z_t]$ if $e_i=0$.
\end{cor}
\begin{proof}
    We first consider the weight forms $(+,+,-,-,0,0;0,0)$. Notice that we can write $f_1=g_1(z_1,z_2)+\sum_{i,j} x_i y_j h_{i,j}$, for some polynomial $g_1$ and $h_{i,j}$ by $\wt f_1 =0$. Thus, $f_1(\mathbf{0},\mathbf{y},\mathbf{z}) = f_1(\mathbf{x},\mathbf{0},\mathbf{z}) = g_1(\mathbf{z})$. Similarly, let $g_2(\mathbf{z})\coloneqq  f_2(\mathbf{0},\mathbf{y},\mathbf{z}) = f_2(\mathbf{x},\mathbf{0},\mathbf{z})$. 
    
    Assume that $g_1$ is a zero polynomial. Pick a point $(\lambda_1,\lambda_2)\neq0$ in $(g_2=0)\subseteq\cc^2$. Then $f_i(x_1,x_2,0,0,\lambda_1,\lambda_2)=g_i(\lambda_1,\lambda_2)=0$ for $i=1,2$, and hence
    \begin{align*}
        \dim \pi_-^{-1}(0:0:0:0:\lambda_1,\lambda_2)
        & = \dim \proj \frac{\cc[x_1,x_2]}{(f_1(x_1,x_2,0,0,\lambda_1,\lambda_2),f_2(x_1,x_2,0,0,\lambda_1,\lambda_2))} \\
        & = \dim \proj \cc[x_1,x_2] \\
        & = 1.
    \end{align*}
    This is absurd since $\pi_-$ is isomorphic away from the origin. Thus, neither $g_1$ nor $g_2$ is $0$. 

    The other cases can be proved similarly.
\end{proof}

\begin{prop}[I]\label{propi}
    Let $\tau=\sum_i a_i - \sum_j b_j - \sum_k e_k$. Then the following are equivalent:
    \begin{enumerate}[label = (\arabic*)]
        \item $-K^-$ is $(\pi_-)$-ample [resp. $(\pi_-)$-trivial];
        \item $K^+$ is $(\pi_+)$-ample [resp. $(\pi_+)$-trivial];
        \item $\tau>0$ [resp. $\tau=0$].
    \end{enumerate}
\end{prop}
\begin{proof}
    This proof directly follows from Corollary $3$ in \cite{Br99}.
\end{proof}

We next examine condition S, which is about the singularities. In fact, we have the classification of terminal singularities on $3$-fold:

\begin{thm}[Mori, Reid, Morrison--Stevens, Koll\'ar--Shepherd-Barron]\label{thms}
    Assume $X$ is a complex $3$-fold. The point $P\in X$ is a terminal singularity if and only if $P$ is smooth or the following forms (allowing permutations of $v_1$, $v_2$, $v_3$, $v_4$):
    \begin{itemize}
        \item[\bf(Q)] a quotient singularity with singular germ $\frac{1}{m}(v_1,v_2,v_3)$, where $m>1$, $\gcd(m,v_1 v_2 v_3)=1$ and $v_1+v_2\equiv0\pmod{m}$;
        \item[\bf(H)] a hyperquotient singularity of type $\frac{1}{m}(v_1, v_2, v_3, v_4;e)$ with $0\in(g=0)\subseteq\cc^4$ is an isolated singularity satisfying one of the conditions (H1-H3):
        \begin{itemize}
            \item[\bf(H1)] $m>1$, $\gcd(m,v_1 v_2 v_3)=1$ and $v_1+v_2\equiv v_4\equiv e \equiv0\pmod{m}$, conditions on $g$ as follows, where $q$ denotes the rank of the quadratic part of $g$ and $w_1$, $w_2$, $w_3$, $w_4$ are eigencooedinates of $\cc^4$.
            \begin{center}
            \begin{tabular}{lllll}
                \hline \hline
                \multicolumn{1}{c}{$m$} &
                \multicolumn{1}{c}{$q$} &
                \multicolumn{1}{c}{$\cc^*$-action} &
                \multicolumn{1}{c}{monomials in $g$} &
                \multicolumn{1}{c}{singularity type} \\
                \hline \hline
                $\geq2$ & $\geq2$ & $\frac{1}{m}(a,-a,b,0;0)$ & $w_1 w_2+h_2(w_3,w_4)$ & $cA/m$\\
                $3$ & $1$ & $\frac{1}{3}(2,1,1,0;0)$ & $w_4^2+w_1^3+h$ & $cD/3$ \\
                $2$ & $2$ & $\frac{1}{2}(1,1,1,0;0)$ & $w_1 ^2+w_4^2+h_2(w_2,w_3)$ & $cAx/2$ \\
                $2$ & $1$ & $\frac{1}{2}(1,1,1,0;0)$ & several possibilities &  \\
                \hline \hline
            \end{tabular}
            \end{center}
            where $h$ denotes a choice of one of the following polynomials $\{w_2^3+w_3^3,w_2^2 w_3+w_1 h_4(w_1,w_3)+h_6(w_1,w_3),w_2^3+w_1 h_4(w_1,w_3)+h_6(w_1,w_3)\}$ and $h_k$ denotes a polynomial of degree at least $k$;
            \item[\bf(H2)] $\frac{1}{4}(3,2,1,1;2)$ or $\frac{1}{4}(1,2,3,3;2)$ and the $2$-jet of $g$ is either $w_1^2+w_3^2$ or $w_1^2+w_3 w_4$;
            \item[\bf(H3)] $m=1$ and $0\in(g=0)$ is a cDV singularity.
        \end{itemize}
    \end{itemize}
\end{thm}

First of all, if $P_1$ in $X^-$, then we have a local isomorphism
\[
(P_1\in X^-) \cong \bigg(0\in\frac{(\bar{f}_1=\cdots=\bar{f}_c=0)}{\frac{1}{a_1}(a_2,\ldots,a_r,-\mathbf{b},\mathbf{0};\mathbf{e})}\bigg)
\]
by direct computation, where $\bar{f}_k\coloneqq f_k|_{x_1=1}$. By the above classification, one sees that a terminal singularity is a quotient of hypersurface of multiplicity $2$. It follows that there is at most one of $\{\bar{f}_k\} _{k=1}^{c}$ has order $2$ and at least $c-1$ of $\{\bar{f}_k\}_{k=1}^{c}$ has order $1$. 

Say that $\bar{f}_2,\ldots,\bar{f}_c$ have linear terms, and note that those are linearly independent. Hence, we may assume $\omega_k\in \bar{f}_k$ for all $k\in\{2,\ldots,c\}$, where $x_1,\omega_2,\ldots,\omega_c,u_1,u_2,u_3,u_4$ is a rearrangement of coordinates $x_1,\ldots,x_r,y_1,\ldots,y_s,z_1,\ldots,z_t$. Then the singular germ of $P_1$ is $\frac{1}{a_1}(\wt u_1,\wt u_2,\wt u_3,\wt u_4;e_1)$. Similarly, we have a local isomorphism near $P\in C^-$ by a similar computation. Thus, we can use the classification to find weights and monomials in $f_i$. 

\section{Classification of GIT Flips and Flops}

In this section, we use numerical conditions to classify GIT flips and flops. Our strategy is as follows: First, we start with a weight form in \cref{propc}. Next, we consider two cases $P_1\notin X^-$ and $P_1\in X^-$. If $P_1\notin X^-$, then we have $x_1^m\in f_k$ for some $k$; if $P_1\in X^-$, then we discuss the singular type of $P_1$. Repeat this process at the points $P_2, \ldots, P_{r+s}$ if necessary. Then we can find all the possibilities. 

In the following discussions, we assume that $X^-\to X\gets X^+$ is a reduced GIT flip or flop, i.e., any polynomial $f_j$ has no linear term. Now, we need the following useful Lemma.

\begin{lem}\label{monomial}
    Assume $f_i$ has no linear term for all $i$. Then there exist a coordinate $\omega_k$ and a non-constant monomial $\mu_k\in\cc[\mathbf{x}]$ such that $\mu_k\omega_k\in f_k$ for all $k$. Moreover, $\omega_i\neq\omega_j$ if $\wt \omega_i\leq0$ and  $i\neq j$. 
    
    Similarly, there are $\eta_j \nu_j \in f_j$ for some coordinates $\eta_j$ and non-constant monomials $\nu_j\in \cc[\mathbf{y}]$. Moreover, $\eta_i\neq\eta_j$ if $\wt\eta_i\geq0$ and $i\neq j$.
\end{lem}
\begin{proof}
    Since $X^-$ has only isolated singularities, a general point $P\in C^-$ is smooth. We may write it as $P=(\lambda_1:\cdots:\lambda_\ell:0:\cdots:0)\in C^-$ with $\ell\leq r$ and $\lambda_k\neq0$ for all $k\leq\ell$. 
    
    Note that the weight function $\wt:M\coloneqq\{x_1^{m_1}\cdots x_\ell ^{m_\ell } \mid m_1,\ldots,m_\ell \in \z \} \cong \z^{\oplus\ell}\to\z$ is a $\z$-module homomorphism. Let $d$ be the positive generator of its image, and then $d=\gcd(a_1,\ldots,a_\ell )$ by Euclidean algorithm. Hence, there is a $x\in M$ such that $\wt x=d$. Let $\widetilde{z}_1,\ldots,\widetilde{z}_{\ell-1}$ be the $\z$-basis of its kernel. By the definition, $\wt \widetilde{z}_k=0$ for all $k$. Then $x,\widetilde{z}_1,\ldots,\widetilde{z}_{\ell-1}$ is a $\z$-basis of $M$, and hence 
    \[
    \cc[x_1^{\pm1},\ldots,x_\ell ^{\pm1}] = \cc[x^{\pm1},\widetilde{z}_1^{\pm1},\ldots,\widetilde{z}_{\ell-1}^{\pm1}].
    \]
    Now, rewrite the coordinate ring of the affine variety $X^-_{1,\ldots,\ell}\coloneqq\cap_{i=1}^\ell X^-_i$ as 
    \[
    \bigg(\frac{\cc[x^{\pm1},\widetilde{z}_1^{\pm1},\ldots,\widetilde{z}_{\ell-1}^{\pm1},x_{\ell +1},\ldots,x_r,\mathbf{y},\mathbf{z}]}{(f_1,\ldots,f_{c})}\bigg)^{\cc^*}.
    \]
    By the direct calculation, $\oo_{X^-}(X^-_{1,\ldots,\ell })$ is isomorphic to 
    \[
    \bigg(\frac{\cc[\widetilde{z}_1^{\pm1},\ldots,\widetilde{z}_{\ell-1}^{\pm1},\ldots,x_r,\mathbf{y},\mathbf{z}]}{(\bar{f}_1,\ldots,\bar{f}_c)}\bigg)^{\z/d\z}, 
    \]
    where $\bar{f}_k=f_k|_{x=1}$. 
    
    Say $(1:\lambda_1':\cdots:\lambda_{k-1}',0:\cdots:0)$ is the coordinates of $P$ under the coordinates system $x,\widetilde{z}_1,\ldots,\widetilde{z}_{\ell-1},x_{\ell+1},\ldots,x_r,\mathbf{y},\mathbf{z}$. Let $\bar{z}_j\coloneqq \widetilde{z}_j-\lambda_j'$, and then the coordinates of $P$ becomes $(1:0:\cdots:0)$ under the coordinates system $x,\bar{z}_1,\ldots,\bar{z}_{\ell-1},x_{\ell+1},\ldots,x_r,\mathbf{y},\mathbf{z}$. Now, $P$ is smooth, so its index and multiplicity are $1$. Thus, $d=1$ and there is a rearrangement $x,\omega_1',\ldots,\omega_c',u_1,u_2,u_3$ of coordinates $x,\bar{z}_1,\ldots,\bar{z}_{\ell-1},x_{\ell+1},\ldots,x_r,\mathbf{y},\mathbf{z}$ such that $\omega'_k$ in $\bar{f}_k$. Therefore, $X^-_{1,\ldots,\ell}$ is an analytic open neighborhood of $\cc^3_{u_1,u_2,u_3}$. 
    
    If $\omega'_k=\bar{z}_j$ for some $j$, then $x^m \widetilde{z}_1^{m_{1}}\cdots\widetilde{z}_{\ell-1}^{m_{\ell-1}}\bar{z}_j\in f_k$ for some $m,m_1,\ldots,m_{\ell-1}\in \z$. Hence, $\lambda_j' x^m \widetilde{z}_1^{m_{1}}\cdots\widetilde{z}_{\ell-1}^{m_{\ell-1}}\in f_k$, and this monomial comes from a non-constant monomial $x_1^{\alpha_1}\cdots x_\ell^{\alpha_\ell}$ for some $\alpha_1,\ldots,\alpha_\ell\in\z_{\geq0}$. In this case, we pick $\omega_k\coloneqq x_i$ and $\mu_k\coloneqq x_1^{\alpha_1}\cdots x_\ell^{\alpha_\ell}/x_i$ for some $i$ such that $\alpha_i\geq1$, so $\mu_k \omega_k\in f_k$. 

    If $\omega'_k\in\{x_{\ell +1},\ldots,x_r,\mathbf{y},\mathbf{z}\}$, then we choose $\omega_k\coloneqq\omega_k'$. By a similar argument, there is a monomial $\mu_k\in\cc[x_1,\ldots,x_\ell]$ such that $\mu_k \omega_k\in f_k$. More precisely, since $f_k$ has no linear term, so $\mu_k$ is not a constant. By our choices of $\omega_k$, if $\wt\omega_i\leq0$, then $\omega_i\neq\omega_j$ for $i\neq j$. 
\end{proof}

From now on, we also assume $a_1\geq a_i$ for all $i$.

\begin{cor}\label{lw}
    Assume $f_i$ has no linear term for all $i$. Then $P_1\in X^-$. Similarly, $P_{r+1}\in X^+$ if $b_1\geq b_j$ for all $j$.
\end{cor}
\begin{proof}
    Suppose on the contrary that $P_1\notin X^-$. Then one of $f_i$ contains $x_1^m$ for some $m\in\z$. By assumption, $f_i$ has no linear term, so $m\geq2$. By \cref{monomial}, there is a monomial $\nu\in\cc[y_1,\ldots,y_s]$ and a coordinate $\eta$ such that $\eta \nu\in f_i$. Now, $\wt \nu < 0$, so
    \[
    a_1 > \wt\eta \geq \wt\eta + \wt \nu = e_i = m a_1 \geq 2a_1,
    \]
    which is a contradiction. Thus, $P_1\in X^-$.
\end{proof}

\begin{cor}\label{isosing}
    Assume $f_i$ has no linear term for all $i$. Then $a_1>a_j$ for all $j\neq1$ if $a_1>1$. 
\end{cor}
\begin{proof}
    Suppose on the contrary that $a_1=a_2>1$. Then any point in the line $\overline{P_1 P_2}$ has index $a_1>1$, so it is a singularity in $V^-$. However, $X^-$ has only isolated singularities, so the intersection of $X^-$ and $\overline{P_1 P_2}$ is a finite set. Now, pick $P=(\lambda_1:\lambda_2:0:\cdots:0)\notin X^-$, then there exists $x_1^\ell x_2^m \in f_j$ for some $j$ and integers $\ell,m\geq0$. Since $f_j$ has no linear term, we get $a_1>e_j=\ell a_1+m a_2\geq 2 a_1$, which is a contradiction. Therefore, $a_1>a_j$ for all $j\neq1$.
\end{proof}

\subsection{Hypersurface Case}

Now, we classify flops of the hypersurface case. In this part, we simplify the notation by setting  $f \coloneqq f_1$, $e\coloneqq e_1$, and $z\coloneqq z_1$.

\begin{thm}\label{thmflop5}
    Any reduced GIT flop has $\cc^*$-action 
    \[
    (1,1,-1,-1,0;0), 
    \]
    and it has the following forms:
    \[
    \begin{tabular}{lll}
        \hline \hline
        \multicolumn{1}{c}{polynomial $f$} &
        \multicolumn{1}{c}{Sing. in $X^\mp$} \\
        \hline \hline
            $x_2 y_2 + x_1 y_1 + z^2 p(z)$ & smooth \\
            $x_2 y_2+x_1 y_1 g(x_1,y_1,z) + z h(x_1,y_1,z)$ & $cA$ \\
        \hline \hline
    \end{tabular}
    \]
    where $p\in \cc[z]$, and $g, h\in (x_1,y_1,z)$, $\wt g = \wt h = 0$. 
\end{thm}

\begin{proof}
    First, we claim that $a_1=1$. Suppose on the contrary that $a_1>1$. Note that $P_1\in X^-$ by \cref{lw}. By \cref{thms}, its singular type is either Q, H1, or H2 when $a_1>1$. Say $\{x_1, w_1, w_2, w_3, w_4\}$ is a permutation of coordinates system with weights $\{a_1, v_1, v_2, v_3, v_4\}$, respectively. When $P_1$ is a quotient singularity, then there is a linear term in $\bar{f}\coloneqq f|_{x_1=1}$. We may assume $x_1^m\omega_4\in f$ for some $m$. Hence, $P_1$ is of type $\frac{1}{a_1}(v_1, v_2, v_3)$. In particular, $\gcd(a_1,v_3)=1$ and $v_1+v_2\equiv0\pmod{a_1}$. On the other hand,
    \begin{align*}
        \tau 
        = a_1 + v_1 + v_2 + v_3 + v_4 - e 
        & = a_1 + v_1 + v_2 + v_3 + v_4 - (m a_1 + v_4) \\
        & \equiv v_3 \pmod{a_1}.
    \end{align*}
    By \cref{propi}, $\tau=0$ , and hence $a_1>1$ divides $v_3$, contradicting to $\gcd(a_1,v_3)=1$. Similarly, for H1 and H2 cases, we have $\tau \not\equiv 0 \pmod{a_1}$, which is a contradiction. Hence, $a_1=1$. 
    
    It follows that $a_i=1$ since we assume that $a_1\geq a_i$ for all $i$. By the same argument, $b_j=-1$ for all $j$. Moreover, the weight form is either $(1,1,-1,-1,0;0)$, $(1,1,1,-1,-1;1)$, or $(1,1,-1,-1,-1;-1)$ by \cref{propc} and \cref{propi}. However, if $e>0$, then $f$ contains a monomial $\mu\in\cc[x_1,x_2,x_3]$ by the proof of \cref{propc}. Now, $\wt \mu = \wt f = 1$, so $\mu=x_i$ for some $i$, which is a contradiction. Also, $e<0$ will lead to a similar contradiction. Therefore, the weight form is $(1,1,-1,-1,0;0)$. 

    Next, we find monomials of $f$ and the possible singularity types of the points in $X^\mp$. By \cref{monomial}, there exists a coordinate $\eta\in\{x_1,x_2,y_1,y_2,z\}$ and a non-constant monomial $\nu\in\cc[y_1,y_2]$ such that $\eta \nu\in f$. This shows that $\wt\eta>0$. It follows that $\eta=x_i$ for some $i$ and $\wt\eta=1$. Hence, $\wt\nu=-1$ and $\nu=y_j$ for some $j$. Without loss of the generality, we assume $f=x_2y_2+g$ for some homogeneous polynomial $g\in\cc[x_1,y_1,z]$. Now, the affine chart $X^-_2$ is $(f|_{x_2=1}=0)\subseteq\cc^4_{x_1,y_1,y_2,z}$. So $X^-_2 = (y_2+g=0) \cong \cc^3_{x_1,y_1,z}$. Hence, $C^-$ is smooth away from $P_1$. 
    
    If $P_1$ is smooth, then $\bar{f}$ contains linear term $\omega\in\{x_2,y_1,y_2,z\}$, i.e., $x_1^m \omega\in f$ for some $m\in\z_{>0}$. Thus, $m=1$ and $\omega=y_j$ for some $j$ since $e=0$. Moreover, $j=1$ since $x_1y_j\in g\in\cc[x_1,y_1,z]$, so we can write $f=x_2y_2+x_1y_1+h$ for some $h\in\cc[z]$. Besides, note that $z^2\mid h$ since $0\in X$ and $f$ have no linear terms. 
    
    Assume $P_1$ is cDV. Recall that
    \[
    C^- =\proj\frac{\cc[x_1,x_2]}{(f(x_1,x_2,0,0,0))} \quad
    \mbox{and} \quad
    C^+ =\proj\frac{\cc[y_1,y_2]}{(f(0,0,y_1,y_2,0))}
    \]
    have of dimension $1$, so $f(x_1,x_2,0,0,0)=f(0,0,y_1,y_2,0)=0$. This shows $g(x_1,0,0)=g(0,y_1,0)=0$, that is $g\in (y_1,z)\cap(x_1,z)=(x_1y_1,z)$. Thus, $g$ has form $x_1 y_1 \widetilde{g}(x_1,y_1,z) + z h(x_1,y_1,z)$ for some monomials $\widetilde{g},h$. Moreover, $\widetilde{g},h\in (x_1,y_1,z)$ because $\bar{f}$ has no linear term. 
\end{proof}

\begin{eg}
    Let $f = x_2 y_2 + z^m + x_1^\ell y_1^\ell$ for some $m,\ell\geq2$ and weights are $(1,1,-1,-1,0;0)$, then we have singularities in the following charts
    \begin{align*}
        & X^-_1 = (x_2 y_2 + z^m + y_1^\ell=0)\subseteq \cc^4_{x_2,y_1,y_2,z}, \\
        & X^+_1 = (x_2 y_2 + z^m + x_1^\ell=0)\subseteq \cc^4_{x_1,x_2,y_2,z}.
    \end{align*}
    $X^\mp$ has only one $cA$ singularity in this example. 
\end{eg}

\subsection{Codimension 2 Case}

We now investigate the codimension $2$ case. We will see that only the weight form $(+,+,-,-,0,0;0,0)$ (resp. $(+,+,+,-,-,0;+,0)$) give reduced GIT flops (resp. flips). The others can not. 

First, we classify the reduced GIT flips and flops with weight form $(+,+,-,-,0,0;0,0)$.

\begin{thm}\label{thmflop6}
    Suppose that the weight of a reduced GIT flips and flops of codimension $2$ is of the form $(a_1,a_2,-b_1,-b_2,0,0;0,0)$. Then it is
    \[
    (1,1,-1,-1,0,0;0,0).
    \]
    Moreover, polynomials $f_1,f_2$ has the following forms:\\
    \begin{center}
    \begin{tabular}{lll}
        \hline \hline 
        \multicolumn{1}{c}{} &
        \multicolumn{1}{c}{monomials in $f_1$} &
        \multicolumn{1}{c}{monomials in $f_2$} \\
        \hline \hline
            $(1)$ & $x_1y_1 + x_2 y_2 + g_1(z_1,z_2)$ & $x_1y_2 + x_2 y_1 + g_2(z_1,z_2)$ \\
            $(2)$ & $x_1y_1 + x_2 y_2 + g_1(z_1,z_2)$ & $x_2y_1 + g_2(z_1,z_2)$ \\
            $(3)$ & $x_1y_1 + g_1(z_1,z_2)$ & $x_2y_2 + g_2(z_1,z_2)$ \\
        \hline \hline
    \end{tabular}
    \end{center}
    where $g_i\in(z_1,z_2)-\{0\}$ for all $i$. 
\end{thm}
\begin{proof}
    First, we claim that the weight form is $(1,1,-1,-1,0,0;0,0)$ and this gives a flop. Since $P_1$ is terminal, one of $\bar{f}_1$ and $\bar{f}_2$ contains some linear term. Without loss of the generality, say $x_1^m y_1 \in f_1$ by $e_1=0$. Notice that the singular germ of $P_1$ is of the form $\frac{1}{a_1}(*,*,0)$ or $\frac{1}{a_1}(*,*,0,0;*)$. By the classification of \cref{thms}. one has $a_1=1$. By the same argument, $a_i=b_j=1$ for all $i,j$. Thus, $\tau=0$.

    Second, we determine monomials in $f_1$ and $f_2$. Note that $g_i\coloneqq f_i(0,0,0,0,z_1,z_2) \neq 0$ for all $i$ by \cref{corc}. Besides, there exist a coordinate $\eta_j$ and a monomial $\nu_j\in\cc[y_1,y_2]$ such that $\eta_j \nu_j\in f_j$ by \cref{monomial}. Notice $e_j=0$ and $\wt \nu_j <0$, so $\eta_j\in\{x_1,x_2\}$. This shows that $\eta_1\neq\eta_2$ by \cref{monomial} and $\wt\eta_j>0$. Moreover, $\nu_j\in\{y_1,y_2\}$ for both $j$ since $e_j=0$. Without loss of generality, we may assume $f_1=x_1y_1+g$ for some $g\in\cc[x_2,y_2,z_1,z_2]$ and $x_2 y_i\in f_2$ for some $i$. 

    Finally, we discuss the following three cases:
    \begin{itemize}
        \item[\bf Case$\bf1$] Both $P_1$ and $P_2$ are smooth.\\
            For $j=1,2$, there exists $x_2^{m_j} \omega_j\in f_j$ for some integers $m_1,m_2$ and coordinates $\omega_1\neq\omega_2$ because $P_2$ is smooth. Notice that $e_j=0$ lead to $m_j=1$ and $\omega_j\in\{y_1,y_2\}$. On the other hand, the coordinate $\omega_1$ in the $f_1|_{x_2=1}=x_1y_1+g(1,y_2,z_1,z_2)$, so we have $\omega_1=y_2$. Moreover, $\omega_2=y_1$ due to $\omega_2\neq\omega_1=y_2$. Again, we may assume $f_1=x_1y_1+x_2y_2+g_1(z_1,z_2)$. Under this assumption, the linear term of $\bar{f}_1=f_1|_{x_1=1}$ is $y_1$, so the linear term of $\bar{f}_2$ contains $y_2$. This implies $x_1y_2\in f_2$ by $e_2=0$.
        \item[\bf Case$\bf2$] $P_1$ is cDV and $P_2$ is smooth.\\
            By the same approach in the previous case, since $P_2$ is smooth, we may assume that $f_1=x_1y_1+x_2y_2+g_1(z_1,z_2)$ and $x_2 y_1\in f_2$. 
        \item[\bf Case$\bf3$] Both $P_1$ and $P_2$ are cDV.\\
            In this case, we claim $x_2y_2\in f_2$. Assume the opposite, that is $x_2y_2\notin f_2$. Then $x_2y_1\in f_2$ since $x_2 y_i\in f_2$ for some $i$. Now, we have $x_1y_2\notin f_2$ and $x_2y_2\notin f_1$ since $x_1y_1\in f_1$, $x_2y_1\in f_2$, and neither $P_1$ nor $P_2$ is smooth. This implies that neither $f_1|_{y_2=1}$ nor $f_2|_{y_2=1}$ has linear terms, which is a contradiction. \qedhere
    \end{itemize}
\end{proof}

Now, we give examples in any cases of \cref{thmflop6}.

\begin{eg}[flop]\label{ex1flop6}
    Let $f_1 = x_1 y_1+x_2^2 y_2^2 + z_1^3$, $f_2 = x_2 y_2+x_1^2 y_1^2 + z_2^3$, and weights are $(1,1,-1,-1,0,0;0,0)$. Then there is one $cA$ singularity in the following charts.
    \begin{align*}
        & X^-_1 = (x_2 y_2+z_2^3+(x_2^2 y_2^2+z_1^3)^2=0)\subseteq \cc^4_{x_2,y_2,z_1,z_2}, \\
        & X^-_2 = (x_1 y_1+z_1^3+(x_1^2 y_1^2+z_2^3)^2=0)\subseteq \cc^4_{x_1,y_1,z_1,z_2}, \\
        & X^+_1 = (x_2 y_2+z_2^3+(x_2^2 y_2^2+z_1^3)^2=0)\subseteq \cc^4_{x_2,y_2,z_1,z_2}, \\
        & X^+_2 = (x_1 y_1+z_1^3+(x_1^2 y_1^2+z_2^3)^2=0)\subseteq \cc^4_{x_1,y_1,z_1,z_2}. 
    \end{align*}
    In this case, $X^\mp$ has two singularities of type $cA$.
\end{eg}

\begin{eg}[flop]\label{ex2flop6}
    Let $f_1=x_1 y_2 - x_2 y_1-z_2^2$, $f_2=x_2 y_2 + x_1^m y_1^m \omega + g(z_1,z_2)$ for some $\omega\in\{1,z_2\}$, and weights are $(1,1,-1,-1,0,0;0,0)$. Then both $X^-$ and $X^+$ have at most one singularity in the following charts
    \begin{align*}
        & X^-_1 = (x_2^2 y_1 + x_2 z_2^2 + y_1^m \omega + g(z_1,z_2)=0)\subseteq \cc^4_{x_2,y_1,z_1,z_2}, \\
        & X^+_1 = (x_1 y_2^2 + y_2 z_2^2 + x_1^m \omega + g(z_1,z_2)=0)\subseteq \cc^4_{x_1,y_2,z_1,z_2}.
    \end{align*}
    for suitable $(m,\omega,g)$. Now, we list some different $(m,\omega,g)$ such that the singularity $P_1$ is of type $cA$, $cD$, $cE_6$, $cE_7$, $cE_8$, respectively.
    \begin{table}[H]
        \centering
        \begin{tabular}{lllc}
            \hline 
            \multicolumn{1}{c}{$(m,\m{\omega})$} &
            \multicolumn{1}{c}{$g$} &
            \multicolumn{1}{c}{$X^-_1$} &
            \multicolumn{1}{c}{Sing. in $X^\mp$} \\
            \hline 
            $(m,\m{1})$ & $z_1^2+z_2^2$ & $(z_1^2+ \len{$x_2^2 y_1$}{z_2^2} + \len{$z_2 y_1^3$}{y_1^m} + x_2( x_2y_1 + z_2^2 ) = 0 )$ & $cA_{\phantom{6}}$ \\
            $(m,\m{1})$ & $z_1^2 + z_2^\ell$ & $(z_1^2 + x_2^2 y_1 + \len{$z_2 y_1^3$}{y_1^m} + \len{$x_2$}{z_2^2} (\len{$x_2y_1$}{z_2^{\ell-2}}+\len{$z_2^2$}{x_2})  = 0 )$ & $cD_{\phantom{6}}$ \\
            $(\m{4},\m{1})$ & $z_1^2 + z_2^3$ & $(z_1^2 + \len{$x_2^2 y_1$}{z_2^3} + \len{$z_2 y_1^3$}{y_1^4}  + x_2( x_2y_1 + z_2^2 ) = 0 )$ & $cE_6$ \\
            $(\m{3},z_2)$ & $z_1^2 + z_2^3$ & $(z_1^2 + \len{$x_2^2 y_1$}{z_2^3} + z_2 y_1^3  + x_2( x_2y_1 + z_2^2 ) = 0 )$ & $cE_7$ \\
            $(\m{5},\m{1})$ & $z_1^2 + z_2^3$ & $(z_1^2 + \len{$x_2^2 y_1$}{z_2^3} + \len{$z_2 y_1^3$}{y_1^5}  + x_2( x_2y_1 + z_2^2 ) = 0 )$ & $cE_8$ \\
            \hline 
        \end{tabular}
        \caption{Examples of GIT flops with cDV for codimension $2$ case}
        \label{flop6cDV}
    \end{table}
    For any example in the above table, $m,\ell \geq3$ and $P_1$ is the isolated singularity in $X^-_1$. Similarly, the unique singularity of $X^+$ is $P_3$.
\end{eg}

\begin{table}[H]
        \centering
        \begin{tabular}{cll}
            \hline 
            \multicolumn{1}{c}{Example} &
            \multicolumn{1}{c}{Sing. in $X^-$} &
            \multicolumn{1}{c}{Sing. in $X^+$} \\
            \hline 
            \ref{ex1flop6} & $cA$, $cA$ & $cA$, $cA$ \\
            \ref{ex2flop6} & cDV & cDV \\
            \hline 
        \end{tabular}
        \caption{Singularity types of examples for GIT flops of codimension $2$}
        \label{ego6}
\end{table}

Next, we discuss the weight form $(+,+,+,-,-,0;+,0)$. In this part, denote $z \coloneqq z_1$. 

\begin{thm}\label{thmflip6}
    Any reduced GIT flip or flop with weight form $(a_1,a_2,a_3,-b_1,-b_2,0;e_1,0)$ ($e_1>0$) has $\cc^*$-action 
    \[
    (a_1,a_2,a_3,-b_1,-a_2,0;a_1-b_1,0), 
    \]
    where $a_1>a_2,a_3,b_1>0$ such that $ a_2 a_3 | (a_1-b_1)$ and $a_1,a_2,a_3,b_1$ are pairwise coprime except $\gcd(a_1,b_1)$ may not be $1$. Moreover, it belongs to one of the following forms.\\
    \begin{center}
    \begin{tabular}{llll}
        \hline \hline
        \multicolumn{1}{c}{} &
        \multicolumn{1}{c}{monomials in $f_1$} &
        \multicolumn{1}{c}{monomials in $f_2$} &
        \multicolumn{1}{c}{$\cc^*$-action} \\
        \hline \hline
             $(1)$ & $x_1y_1+h(x_2^{a_3},x_3^{a_2})+x_i y_2^k $ & $x_2y_2 + z^m +x_jy_1^\ell$ & $(a_1,\m{1},a_3,-\m{1},-\m{1},0;a_1-\m{1},0)$ \\
             $(2)$ & $x_1y_1+h(x_2^{a_3},x_3^{a_2})+x_i y_2^k $ & $x_2y_2 + z^m$ & $(a_1,\m{1},a_3,-\m{b_1},-\m{1},0;a_1-\m{b_1},0)$ \\
             $(3)$ & $x_1y_1+h(x_2^{a_3},x_3^{a_2}) $ & $x_2y_2 + z^m+x_j y_1^\ell$ & $(a_1,\m{a_2},a_3,-\m{1},-\m{a_2},0;a_1-\m{b_1},0)$ \\
             $(4)$ & $x_1y_1+h(x_2^{a_3},x_3^{a_2}) $ & $x_2y_2 + z^m$ & $(a_1,\m{a_2},a_3,-\m{b_1},-\m{a_2},0;a_1-\m{b_1},0)$ \\
        \hline \hline
    \end{tabular}
    \end{center}
    for some integers $k,\ell\geq1$, $i\in\{1,3\}$, $j\in\{2,3\}$, $m\geq2$, and polynomial $h(u,v)$ such that $f_1(x_1,x_2,x_3,0,0,0)=h(x_2^{a_3},x_3^{a_2})$ and the ring $\cc[x_1,x_2,x_3]/h(x_2^{a_3},x_3^{a_2})$ is reduced. More precisely, $P_1$ is the unique singularity of $X^-$, which is $cA/a_1$ type.
\end{thm}

\begin{proof}
    First, we show that $P_1\in X^-$ is H1 type. Say $e_1>0$ and $f_1$ contains a monomial in $\cc[x_1,x_2,x_3]$ by the proof of Proposition S. This implies that $e_1\geq2$ since we assume $f_j$ has no linear terms. By \cref{propi}, $3a_1\geq a_1+a_2+a_3\geq b_1+b_2+e_1\geq4$, hence, $a_1>1$. Thus, $P_1$ is neither smooth nor H3. By \cref{isosing}, $a_1>a_i$ for all $i\neq1$. 
    
    When $P_1$ is a quotient singularity or H2, then one of $\bar{f}_1$ and $\bar{f}_2$ contains $z$; otherwise, the singular germ of $P_1$ is of type $\frac{1}{a_1}(*,*,0)$ or $\frac{1}{a_1}(*,*,*,0;*)$, which is a contradiction. Assume $x_1^\alpha z \in f_i$ for some $i$ and $\alpha\geq1$. Now, $e_i=\alpha a_1>0$, so this say $x_1^\alpha z \in f_1$. If $P_1$ is a quotient singularity, then we may assume that $x_1^\beta y_1\in f_2$ for some $\beta\geq1$. By \cref{propi}, we have
    \[
    3 a_1 > a_1+a_2+a_3\geq b_1+b_2+e_1 = b_2+(\alpha+\beta)a_1 \geq b_2 + 2 a_1.
    \]
    This implies that $0<a_2-b_2, a_3-b_2<a_1$ and $a_1<a_2+a_3<2a_1$. However, the singular germ of $P_1$ is $\frac{1}{a_1}(a_2,a_3,-b_2)$, and then $a_1$ divides $a_2+a_3$, $a_2-b_2$, or $a_3-b_2$, which is a contradiction. If $P_1$ is H2, then the singular germ of $P_1$ becomes $\frac{1}{a_1}(a_2,a_3,-b_1,-b_2;e_2=0)$, which is a contradiction. Hence, $P_1$ is H1. 
    
    Next, we determine monomials in $f_1$ and $f_2$. By \cref{corc}, there exists $z^m\in f_2$ for some $m\geq2$. Note that one of $f_1$ and $f_2$ contains $x_1^\ell\omega$ for some $\ell\in\z_{>0}$ and a coordinate $\omega\neq x_1$. By the weight reason, the possible situations are $x_1^\ell y_i\in f_1$, $x_1^\ell x_i\in f_1$, $x_1^\ell z\in f_1$, or $x_1^\ell y_i\in f_2$. 
    
    If $x_1^\ell y_1\in f_1$, then the singular germ of $P_1$ is $\frac{1}{a_1}(a_2,a_3,-b_2,0;e_2=0)$ and $e_1= \ell a_1 - b_1$. Thus, $a_2 + a_3 \geq b_2 + (\ell - 1) a_1$ by \cref{propi}. This shows that $\ell$ is $1$ or $2$. If $\ell=2$, then $a_1<a_2+a_3<2a_1$ and $0<a_3-b_2<a_1$, which is a contradiction. Therefore, $\ell=1$. By the same argument, if $\omega\neq y_i$ for some $i$, then we get contradictions, which are similar to the $\ell=2$ case. Thus, without loss of generality, we may assume $f_1=x_1 y_1-g(x_2,x_3,y_2,z)$.
    
    Now, we have
    \[
    X^-_1 = (f_2(1,x_2,x_3,g,y_2,z)=0)\subseteq \cc^4_{x_2,x_3,y_2,z}/\textstyle\frac{1}{a_1}(a_2,a_3,-b_2,0),
    \]
    and $P_1$ is H1 with singular germ $\frac{1}{a_1}(a_2,a_3,-b_2,0;0)$. This implies that
    \begin{align*}
        \tau
        = a_1+a_2+a_3-b_1-b_2-e_1
        & = a_1+a_2+a_3-b_1-b_2-(a_1-b_1)\\
        & = a_2+a_3-b_2\\
        & \not\equiv0\pmod{a_1}.
    \end{align*}
    By \cref{propi}, $\tau>0$ and $X^-\to X\gets X^+$ is a flip. 
    
    More precisely, we claim that $P_1$ is $cA/a_1$ and $x_i y_2\in f_2$ for some $i\neq1$. By \cref{thms}, $P_1$ is either $cA/a_1$, $cD/3$, or $a_1=2$. Notice that $a_1=e_1+b_1\geq3$ since $x_1 y_1\in f_1$ and $e_1>1$. 
    
    If $P_1$ is $cA/a_1$, then there are coordinates $\omega_1\neq\omega_2$ in $\{x_2,x_3,y_2,z\}$ such that $\omega_1\omega_2\in f_2(1,x_2,x_3,g,y_2,z)$ and $\wt\omega_1\equiv-\wt\omega_2\not\equiv0\pmod{a_1}$. This shows that either $x_1^\ell \omega_1\omega_2\in f_2$ for some $\ell\geq0$, or $\omega_1\omega_2\in g$ and $x_1^\ell y_1\in f_2$ for some $\ell\geq0$. By similar arguments of $x_1 y_j\in f_1$ and conditions $e_2=0$ and $0<\tau=a_2+a_3-b_1<2a_1$, we conclude that $x_i y_2\in f_2$ for some $i\neq1$. 
    
    For the $cD/3$ case, we also find some possible monomials in $f_2$ by \cref{thms}, and all situations are impossible by weight and degree reason. Without loss of generality, we say $x_2 y_2\in f_2$, and then $b_2=a_2$.

    Finally, we prove that any $P\in X^- -\{P_1\}$ is smooth. Note that $f_1(x_1,x_2,x_3,0,0,0)=-g(x_2,x_3,0,0)\in\cc[x_2,x_3]$ and denoted by $h_0(x_2,x_3)$. 
    
    If $P=P_2\in X^-$, then $x_2^m\notin h_0$, i.e., $x_3\mid h_0$. On the other hand, $x_3^2\nmid h_0$ since $C^-$ is reduced. This shows that $x_3\in h_0|_{x_2=1}$, and hence $x_3\in f_1|_{x_2=1}$. Thus, $P_2$ is smooth since $X^-_2=\cc^3_{x_1,y_1,z}/\frac{1}{a_2}(a_1,-b_1,0)$. Similarly, if $P=P_3\in X^-$, then $x_2\in f_1|_{x_3=1}$. Now, the singular germ of $P_3$ is $\frac{1}{a_3}(a_1,-b_1,-a_2)$ or $\frac{1}{a_3}(a_1,-b_1,-a_2,0;e_1)$, so it is not H2. Note $a_3\mid a_1-b_1-a_2$ since $x_1y_1+x_2\in f_1|_{x_3=1}$. Thus, $P_3$ is neither a quotient singularity nor H1. Moreover, $P_3$ is not cDV; otherwise, we can show that $P_3$ is not isolated, which is a contradiction. Hence, $P_3$ is smooth by \cref{thms}. Besides, if $P_i\notin X^-$, then $x_i^m\in f_1$ for some $m$ and hence $a_i\mid e_1$. 
    
    From now on, we have $P_i$ is smooth if and only if $P_i\in X^-$, $a_i\mid e_1$ ($i=2,3$), and $\gcd(a_1,a_2a_3)=1$ (since $P_1$ is H1). It is sufficient to show that $\gcd(a_2,a_3)=1$ and the rest points are smooth. Let $d=\gcd(a_2,a_3)$. If one of $P_2$ and $P_3$ is in $X^-$, then $\gcd(a_2,a_3)=1$. If $P_2,P_3\notin X^-$, then there exists $P\in\overline{P_2P_3}\cap X^-$ by the dimension reason. Notice that the singular germ of $P$ is one of forms $\frac{1}{d}(*,*,0,0;0)$ or $\frac{1}{d}(*,*,0)$ since $\wt z=e_2=0$ and $d$ divides $a_2$, $a_3$, and $e_1$. This shows that $P$ is smooth or cDV, so $\gcd(a_2,a_3)=d=1$. 
    
    Similarly, the rest points in $X^-$ are of index $1$. $P$ is not cDV by the changing coordinates and similar arguments of $P_2$ and $P_3$. Besides, we observe that $x_2^\alpha x_3^\beta\in f_1$ if and only if $a_3\mid \alpha$ and $a_2\mid \beta$ since $\gcd(a_2,a_3)=1$ and $a_2 a_3\mid e_1$. Hence, $h_0=h(x_2^{a_3},x_3^{a_2})$ for some polynomial $h(u,v)$ and $h$ is homogeneous with respect to degree.
\end{proof}

\begin{rmk}
In the above proof, since $C^-$ is reduced, $C^-$ has irreducible components $\Gamma_i^-\coloneqq (L_i(x_2^{a_3},x_3^{a_2})=0)\subseteq\pp(a_1,a_2,a_3)$, where $i=1,\ldots,e_1/(a_2a_3)$ and $L_i$ is a polynomial with degree $1$ such that $h=\Pi_i L_i$.

\begin{figure}[H]
    \centering
    \begin{tikzpicture}
        \def\ul{1}
        \draw[scale = \ul, fill = black] (-1.8,-1) circle [radius=.08];
        \draw[scale = \ul] (-2.2,-1.4) node {$P_2$};
        \draw[scale = \ul, fill = black] (1.8,-1) circle [radius=.08];
        \draw[scale = \ul] (2.2,-1.4) node {$P_3$};
        \draw[scale = \ul, fill = black] (0,2) circle [radius=.08];
        \draw[scale = \ul] (0,2.5) node {$P_1$};
        \draw[scale = \ul] (-2.2,-1) -- (2.2,-1); 
        \draw[scale = \ul] (-1.92,-1.2) -- (.12,2.2); 
        \draw[scale = \ul] (1.92,-1.2) -- (-.12,2.2);
        \draw[scale = \ul, red] (-1.6,-1.2) -- (.1,2.2);
        \draw[scale = \ul, red] (-1.28,-1.2) -- (.08,2.2);
        \draw[scale = \ul, red] (-.96,-1.2) -- (.06,2.2);
        \draw[scale = \ul, red] (-.64,-1.2) -- (.04,2.2);
        \draw[scale = \ul, red] (.15,.3) node {...}; 
        \draw[scale = \ul, red] (1.28,-1.2) -- (-.08,2.2);
        \draw[scale = \ul, red] (1.6,-1.2) -- (-.1,2.2);
        \draw[scale = \ul, red] (0,-1.5) node {$C^-=\cup_i \Gamma_i^-$}; 
    \end{tikzpicture}
    \caption{$C^-:h(x_2^{a_3},x_3^{a_2})=0$ in $\pp(a_1,a_2,a_3)$}
    \label{flip6curve}
\end{figure}
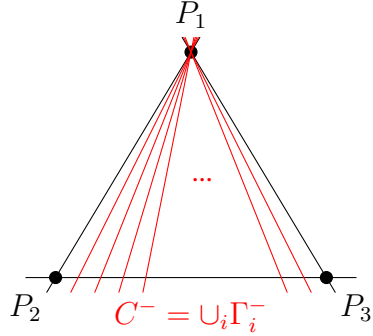
\end{rmk}

Again, we provide examples for all cases of \cref{thmflip6}.

\begin{eg}[flip]\label{ex1flip6}
    Let $f_1=x_1 y_1+ x_1 y_2 + x_2 x_3(x_2^{a-3}+x_3^{a-3}) $ and $f_2=x_2 y_2 + z^m + x_3 y_1 $, and weights are $(a,1,1,-1,-1,0;a-1,0)$ for some $a\geq3$. In this example, the unique singularity is $P_1$ in the chart
    \begin{align*}
        & X^-_1 = (x_2 y_2 + z^m - x_3g = 0)\subseteq \cc^4_{x_2,x_3,y_2,z}/\textstyle\frac{1}{a}(1,1,-1,0), 
    \end{align*}
    where $g=y_2+x_2 x_3(x_2^{a-3}+x_3^{a-3})$. 
\end{eg}

\begin{eg}[flip]\label{ex2flip6}
    Let $f_1=x_1 y_1+ x_2 x_3+ x_3 y_2^{b+1}$ and $f_2=x_2 y_2 + z^m + x_3^b y_1 $, and weights are $(b+2,1,1,-b,-1,0;2,0)$ for some $b\geq1$. Then charts with singularities are
    \begin{align*}
        & X^-_1 = (x_2 y_2 + z^m - x_3g = 0)\subseteq \cc^4_{x_2,x_3,y_2,z}/\textstyle\frac{1}{b+2}(1,1,-1,0) \\
        & X^+_1 = (x_2 y_2 + z^m - x_3g = 0)\subseteq \cc^4_{x_2,x_3,y_2,z}/\textstyle\frac{1}{a}(1,1,-1,0;0),
    \end{align*}
    where $g=x_3^b(x_2 + y_2^{b+1})$. 
\end{eg}

\begin{eg}[flip]\label{ex3flip6}
    Let $f_1=x_1 y_1 + x_2 ( x_2+x_3^a) $ and $f_2=x_2 y_2 + z^m + x_3 y_1 $, and weights are $(2a+1,a,1,-1,-a,0;2a,0)$ for some $a\geq1$. Then we have charts with singularities
    \begin{align*}
        & X^-_1 = (x_2 y_2 + z^{m\phantom{2}} - x_3g = 0)\subseteq \cc^4_{x_2,x_3,y_2,z}/\textstyle\frac{1}{2a+1}(a,1,-a,0;0), \\
        & X^+_2 = (x_1 y_1 + z^{2m} + x_3h = 0)\subseteq \cc^4_{x_1,x_3,y_2,z}/\textstyle\frac{1}{a}(1,1,-1,0;0), 
    \end{align*}
    where $g = x_2 ( x_2+x_3^a)$ and $h = x_3^a y_1 + x_3^{a-1} z^m + x_3 y_1^2 + 2y_1 z^m $.
\end{eg}

\begin{eg}[flip]\label{ex4flip6}
    Let $f_1=x_1 y_1 - ( x_2^2 - x_3^3) $ and $f_2=x_2 y_2 + z^m + x_3^b y_1^2 $, and weights are $(b+6,3,2,-b,-3,0;6,0)$ for some $b\geq1$ such that $\gcd(b,6)=1$. Then we have    
    \begin{align*}
        & X^-_1 = (x_2 y_2 + z^{m\phantom{2}} + x_3g = 0)\subseteq \cc^4_{x_2,x_3,y_2,z}/\textstyle\frac{1}{b+6}(3,2,-3,0;0), \\
        & X^+_1 = (x_2 y_2 + z^{m\phantom{2}} + x_3^b\phantom{g} = 0)\subseteq \cc^4_{x_2,x_3,y_2,z}/\textstyle\frac{1}{b}(3,2,-3,0;0), \\
        & X^+_2 = (x_1 y_1 + z^{2m} + x_3h = 0)\subseteq \cc^4_{x_1,x_3,y_2,z}/\textstyle\frac{1}{3}(b,2,-b,0;0), 
    \end{align*}
    where $g = x_3^{b-1}(x_2^3-x_3^2)^2 $ and $ h = x_3^2 - 2 x_3^{b-1} y_1^2 z^m - x_3^{2b-1} y_1^4$.
\end{eg}

\begin{table}[H]
        \centering
        \begin{tabular}{cll}
            \hline 
            \multicolumn{1}{c}{Example} &
            \multicolumn{1}{c}{Sing. in $X^-$} &
            \multicolumn{1}{c}{Sing. in $X^+$} \\
            \hline 
            \ref{ex1flip6} & $cA/a$ & smooth \\
            \ref{ex2flip6} & $cA/(b+2)$ & $cA/b$\\
            \ref{ex3flip6} & $cA/(2a+1)$ & $cA/a$ \\
            \ref{ex4flip6} & $cA/(b+6)$ & $cA/b$, $cA/3$ \\
            \hline 
        \end{tabular}
        \caption{Singularity types of examples for GIT flips of codimension $2$}
        \label{egi6}
\end{table}

Finally, we prove that the rest weight forms have no new examples. 

\begin{prop}\label{propc6}
    Any reduced GIT flip or flop of codimension $2$ has one of the following weight forms
    \[
    (+,+,+,-,-,0;+,0), \mbox{ or } (+,+,-,-,0,0;0,0).
    \]
\end{prop}
\begin{proof}
    By \cref{propc}, we need to show that the following weight forms 
    \[
    \begin{array}{cc}
        (+,+,-,-,\m{0},\m{0};+,\m{-}), \;\; (+,+,-,-,-,\m{0};-,\m{0}), & (+,+,-,-,-,-;-,-), \\ 
        (+,+,+,-,\m{-},\m{-};+,\m{-}), \;\; (+,+,+,+,-,\m{-};+,\m{+})\phantom{,} & 
    \end{array}
    \]
    have no new examples. For every case, we use our strategy, i.e., similar arguments of proof of \cref{thmflop5}, \cref{thmflop6}, and \cref{thmflip6}, to get contradictions. In this proof, we only discuss $(+,+,-,-,0,0;+,-)$ and $(+,+,-,-,-,0;-,0)$ cases.
    
    For the $(+,+,-,-,0,0;+,-)$ case, without loss of generality, say $x_1^{m}\omega\in f_\ell$ for some $m\geq1$, coordinate $\omega\neq x_1$, and $\ell=1,2$, since $P_1$ is terminal. On the other hand, we may assume $\eta \nu \in f_\ell$ for some monomials $\nu\in \cc[y_1,y_2]$ and coordinate $\eta$ by \cref{monomial}. Now, $m a_1 +\wt\omega=e_1=\wt\eta + \wt \nu < a_1 $ since $a_1\geq\wt\eta$ and $\wt \nu<0$. This shows that $\omega=y_j$ for some $j$, and hence the singular germ of $P_1$ is $\frac{1}{a_1}(*,*,0)$ or $\frac{1}{a_1}(*,*,0,0;*)$. Therefore, $a_1=1$ by \cref{thms}. Similarly, $e_j<a_1=1$ for all $j$ by \cref{monomial}, contradicts that one of $f_1$ and $f_2$ has positive weight. 

    For the $(+,+,-,-,-,0;-,0)$ case, denote $z\coloneqq z_1$. By \cref{monomial}, say that $\mu_i \omega_i \in f_i$ for some monomials $\mu_i\in \cc[x_1,x_2]$ and coordinates $\omega_i$. Since $e_i\leq0$, we may assume $\mu_i y_i \in f_i$ for some monomials $\mu_i\in \cc[x_1,x_2]$. Thus, the inequality of \cref{propi} becomes $a_1 + a_2 \geq b_3 + \wt (\mu_1 \mu_2)$. Hence, there are integers $m_j$ such that $\mu_j=x_2^{m_j}$ for all $j$ and $a_1>a_2$; otherwise, $\wt (\mu_1 \mu_2)\geq a_1+a_2$, which is a contradiction. Similarly, if both $\bar{f}_1$ and $\bar{f}_2$ contains independent linear terms, then $b_1+b_2+b_3+e_1+e_2>a_1+a_2$, which is a contradiction. Thus, we may assume $P_1$ is a hyperquotient singularity and $x_1^m y_2 \in f_1$. Then the singular germ of $P_1$ is $\frac{1}{a_1}(a_2,-b_1,-b_3,0;e_2)$. Thus, $P_1$ is neither H2 nor H3, since the weights in the singular germ of $P_1$ contain $0$ and $a_1>a_2\geq1$. Now, $a_1\mid e_2 = m_2 a_2 -b_2$ and $\gcd(a_1,a_2)=1$, so $a_1\nmid b_2$. This implies $a_1\nmid e_1 = m a_1 - b_2$, and hence $e_1\neq0$ and $e_2=0$. By \cref{propi}, we obtain 
    \[
    a_1+a_2\geq b_1+b_3+(e_1+b_2)=b_1+b_3+m a_1\geq b_1+b_3+a_1, 
    \]
    so $0<a_2-b_1,a_2-b_3,b_1+b_3<a_1$. Hence, $P_1$ is not H1. To sum up, $P_1$ is not terminal, so this weight form is not allowed. 
\end{proof}

\subsection{Higher Codimension Cases}

In this section, we claim that there are no examples in codimension $\geq3$, i.e., \cref{thmhc}. First, we prove the following lemma.

\begin{lem}\label{lemhc}
    If $n>6$, then $a_1>a_i$ for all $i\neq1$.
\end{lem}
\begin{proof}
    By \cref{isosing}, it is sufficient to show that $a_1>1$. Suppose on the contrary that $a_1=1$. Then $a_i=1$ for all $i$. For any $\ell$, there exist coordinate $\eta$ and $\nu\in\cc[\mathbf{y}]$ such that $\eta \nu\in f_\ell$ by \cref{monomial}, so $e_\ell=\wt\eta+\wt \nu\leq 1 - \min_j b_j\leq0$. This shows $r=2$ by the proof of \cref{propc}. By \cref{monomial} and $e_\ell\leq0$, there exist $\mu_\ell\in\cc[x_1,x_2]$ and distinct $\omega_\ell\in\{y_1,\ldots,y_r\}$ such that $\mu_\ell \omega_\ell\in f_\ell$ for all $\ell$. Without loss of generality, we may assume $\mu_\ell y_\ell\in f_\ell$ for all $\ell$ and the weights are $(a_1,a_2,-b_1,\ldots,-b_{c},-v_1,-v_2)$, where $v_1, v_2\geq0$. By \cref{propi}, we have $2=a_1+a_2\geq v_1+v_2+\sum_i\wt \mu_j\geq c=n-4$ since $\mu_j$ is in $(x_1,x_2)-\{0\}$. This implies $n\leq6$, which is a contradiction. 
\end{proof}

From now on, we assume an additional condition $b_1\geq b_j$ for all $j$, then $\widetilde{P}_1\coloneqq P_{r+1}\in X^+$ by \cref{lw}. 

\begin{prop}\label{thmhc3}
    Any GIT flip or flop of codimension $>3$ is not reduced.
\end{prop}
\begin{proof}
    Assume $n\geq8$, and then the $4$-fold $A$ is given by at least $4$ polynomials $f_j$. Since $P_1$ is terminal, there are $n-5$ distinct coordinates $\omega_j\neq x_1$ such that $\omega_j\in \bar{f}_j$. Without loss of generality, say $x_1^{\ell_j}\omega_j\in f_j$ for $j\geq2$, where $\ell_j\geq1$. Similarly, we may assume that there are distinct $\eta_j\in\{x_i,y_j,z_k\}_{i,j,k}$ and $m_j\geq1$ such that $y_1^{m_j}\eta_j\in f_j$ for all $j\geq3$. Then we have
    \[
    a_1-b_1 
    \leq \ell_3 a_1 + \wt \omega_3
    = e_3
    = \wt \eta_3 - m_3 b_1 
    \leq a_1 - b_1,
    \]
    so $\wt \eta_3=a_1$. Similarly, $\wt \eta_4=a_1$. Note $\eta_3\neq\eta_4$, so one of them is not $x_1$, contradict that $a_1>\wt \omega$ for all $\omega\neq x_1$. 
\end{proof}

Finally, we work on the codimension $3$ case.  

\begin{prop}\label{thmhce3}
    Any GIT flip or flop of codimension $3$ is not reduced.
\end{prop}
\begin{proof}
    We will prove this by verifying the type of $P_1$. Now, $P_1$ is neither smooth nor cDV, since $a_1>1$. Similar to the proofs of \cref{lemhc} and \cref{thmhc3}, we may assume that $b_1>b_j$ for all $j\neq1$ and neither $f_3|_{x_1=1}$ nor $f_1|_{y_1=1}$ has linear terms. Without loss of generality, let $x_1^\ell\omega\in f_1$, $x_1 y_1\in f_2$, $\eta y_1^m\in f_3$. Hence, neither $P_1$ nor $\widetilde{P}_1$ is a quotient singularity. 
    
    By \cref{monomial}, we may assume $\mu_i\omega_i\in f_i$ for some monomials $\mu_i\in \cc[\mathbf{x}]$ and coordinates $\omega_i$ such that $\{\omega_i\}_{\wt\leq0}$ are distinct. Similarly, we suppose $\eta_j \nu_j \in f_j$ for some monomials $\nu_j\in \cc[\mathbf{y}]$ and coordinates $\eta_j$ such that $\{\eta_j\}_{\wt\geq0}$ are distinct. Note $\wt\eta\geq\wt \mu_3>0$ since $-m b_1\leq b_1\leq \wt\omega_3$. Say $\eta=x_2$ and $\omega=y_2$ by similar arguments, i.e.,
    \[
    \begin{aligned}
        & x_1^\ell y_2 &\hspace{-1em}&+ \mu_1 \omega_1  + \eta_1 \nu_1 \in f_1, \\
        & x_1 y_1  &\hspace{-1em}&+ \mu_2 \omega_2 + \eta_2 \nu_2 \in f_2, \\
        & x_2 y_1^m &\hspace{-1em}&+ \mu_3 \omega_3 + \eta_3 \nu_3 \in f_3.
    \end{aligned}
    \]
    
    Now, $e_1 > e_2 > e_3$ by $\ell,m\geq1$, $-b_2 > -b_1$, and $a_1>a_2$. Rewrite the coordinates as $x_1$, $x_2$, $y_1$, $y_2$, $u_1$, $u_2$, $u_3$, and set $v_j\coloneqq \wt u_j$, then the singular germs of $P_1$ and $\widetilde{P}_1$ are
    \[
    \frac{1}{a_1}(a_2,v_1,v_2,v_3;e_3) \mbox{ and }
    \frac{1}{b_1}(-b_2,v_1,v_2,v_3;e_1),
    \]
    respectively. 
    
    Next, we claim that $P_1$ is not H2. If not, then $a_1=4$ and we will show that $s+r+t=n\leq6$. Now, $v_j\neq0$ for all $j$, so $t=0$. 
    
    If $s\geq4$, then $e_2,e_3<0$ by the proof of \cref{propc} and $e_1>e_2>e_3$. Since $b_1 = a_1 - e_2 > a_1 = 4$, $\widetilde{P}_1$ is not H2. Moreover, $b_1 > |\wt \omega| > 0$ for any coordinates $\omega$ by $t=0$ and $b_1>a_1$, so $b_1$ cannot divide any weights of coordinates except $-b_1$ and hence $\widetilde{P}_1$ is not H1. To sum up, $\widetilde{P}_1$ is not terminal, which is a contradiction. 
    
    If $r\geq4$, then $e_1,e_2>0$ for the same reason. Now, $a_1-b_1=e_2\geq2$ since we assume that $f_j$ has no linear terms. Thus, $b_1=2$ since $2 = a_1 - 2 \geq b_1 > 1$. However, the weight of $f_3$ gives us
    \[
    1 = (a_1-1) - b_1 \geq a_2 - m b_1 = e_3 = \wt \mu_3 + \wt\omega_3\geq 1 - b_1 = -1, 
    \]
    contradict to $e_3\equiv2\pmod{4}$. Thus, $n=r+s+t\leq 3+3+0=6$, which is a contradiction. 

    Finally, we show that $P_1$ is not H1. If $P_1$ is H1, than we may assume $\gcd(a_1,v_1 v_2)=1$ and $a_1\mid v_3$ by $a_1\nmid a_2$. Since $a_1$ divides both $e_3$ and $v_3$, we say $e_3=-q a_1$ and $v_3=-q' a_1$ for some $q,q'\geq0$ by $e_3 , \wt\omega < a_1$ ($\omega\neq x_1$); hence, $\gcd(a_1,v_1 v_2)=1$. This shows that $t\leq1$. 
    
    When $t=0$, then say $y_3\coloneqq u_3$. This implies $b_1>b_3=q'a_1\geq a_1$, and hence $b_1>|\wt\omega|>0$ for all coordinates $\omega\neq y_1$. Thus, $\widetilde{P}_1$ is H2 and $b_1=4$. On the other hand, notice that $e_3<e_2=a_1-b_1<0$, so $r\leq3$. Thus, $s=7-r-t\geq4$. Since neither $f_1$ nor $f_2$ contains linear terms and one of $f_1,f_2$ contains monomial in $\cc[\mathbf{y}]$, so $e_2\leq -2$. Now, $4=b_1\geq a_1+2 >3$, so $a_1=2$. However, notice that $-q a_1 = e_3 = a_2-mb_1 = a_2-4m$, so $a_1\mid a_2$, which is a contradiction. 
    
    For $t=1$ case, denote $z\coloneqq z_1=v_3$ and we only need to consider $3$ weight forms $(+,+,-,-,-,-,0)$, $(+,+,+,-,-,-,0)$, and $(+,+,+,+,-,-,0)$. For the weight form $(+,+,-,-,-,-,0)$, we will get that $a_1\nmid a_2+v_1$, $a_1\nmid a_2+v_2$, and $a_1\nmid v_1+v_2$, so $P_1$ is not H1, which is a contradiction. For the $(+,+,+,-,-,-,0)$ case, we claim that there are infinity many contracted curves $\pi_-^{-1}(0:\cdots:0:\lambda)\subseteq X^-$ or there are infinity many contracted curves $\pi_+^{-1}(0:\cdots:0:\lambda)\subseteq X^+$, contradict to condition C of flips and flops. For the $(+,+,+,+,-,-,0)$ case, we will obtain that $C^+$ is not reduced, which is a contradiction. 
\end{proof}

\begin{proof}[Proof of \cref{thmhc}]
    This is proved by \cref{thmhc3} and \cref{thmhce3}. 
\end{proof}

\nocite{KM98,Mat02,MFK12,Muk03,Mor85,Rei87,Har77}
\bibliographystyle{amsalpha}
\addcontentsline{toc}{section}{References}
\bibliography{References.bib}

\end{document}